\documentclass[12pt,draft,oneside]{article}
\usepackage[cp1251]{inputenc}
\usepackage[english,russian,ukrainian]{babel}
\usepackage{amsmath}
\usepackage{amsthm}
\usepackage{amssymb}
\usepackage{amsfonts}
\usepackage{latexsym}
\usepackage{calc}
\usepackage{indentfirst}
\newtheorem{proposition}{Proposition}
\newtheorem*{proposition*}{Proposition}
\newtheorem{theorem}{Theorem}
\newtheorem*{theorem*}{Theorem}
\newtheorem{remark}{Remark}

\newtheorem*{definition*}{Definition}
\newtheorem{lemma}{Lemma}
\newtheorem*{lemma*}{Lemma}
\newtheorem{corollary}{Corollary}
\newtheorem*{corollary*}{Corollary}

\newcommand{\sep}{/\kern-2pt/ }

\allowdisplaybreaks

\begin{document}
 
 \begin{center}
 	\Large\bf ASYMPTOTIC ESTIMATES OF ENTIRE FUNCTIONS OF BOUNDED $\mathbf{L}$-INDEX IN JOINT VARIABLES
 \end{center}
 
 \begin{center}
 	\large\bf \MakeUppercase{A.\ I.\ BANDURA, O.\ B.\ SKASKIV}
 \end{center}

  		\selectlanguage{english}
 \vspace{20pt plus 0.5pt} {\abstract{ \noindent A.\ I.\ Bandura,  O.\ B.\ Skaskiv, \ 
 		\textit{Asymptotic estimates of entire functions of bounded $\mathbf{L}$-index in joint variables} \vspace{3pt} 
 		
 		In this paper, there are obtained growth estimates of  entire in $\mathbb{C}^n$ function of bounded $\mathbf{L}$-index in joint variables.
 		They describe the behaviour of maximum modulus of entire function on a skeleton in a polydisc by 
 		behaviour of the function $\mathbf{L}(z)=(l_1(z),\ldots,l_n(z)),$ where for every 
 		$j\in\{1,\ldots, n\}$ \ $l_j:\mathbb{C}^n\to \mathbb{R}_+$ is a continuous function.
 		We generalised known results of  W. K. Hayman, M. M. Sheremeta, A. D. Kuzyk, M. T. Borduyak, T. O. Banakh and V. O. Kushnir
 		for a wider class of functions $\mathbf{L}.$ 	
 	 	One of our estimates is sharper even for entire in $\mathbb{C}$ functions of bounded $l$-index than Sheremeta's estimate.

 }} 
 
 
 
 \vskip10pt
 
 \allowdisplaybreaks 
 
 \vskip10pt
 
\section{Introduction} Let $l:\mathbb{C}\to \mathbb{R}_+$ be a fixed positive continuous function, where $\mathbb{R}_+=(0,+\infty).$ An entire function $f$ is said to be of bounded $l-$index \cite{vidlindex} if there exists an integer $m,$ independent of $z,$ such that for all $p$ and all $z\in\mathbb{C}$ $\frac{|f^{(p)}(z)|}{l^p(z)p!}\leq \max\{\frac{|f^{(s)}(z)|}{l^s(z)s!}\colon 0\leq s\leq m\}.$
	The least such integer $m$ is called the $l-$index of $f(z)$  and is denoted by $N(f,l).$ If $l(z)\equiv 1$ then we obtain the definition of  function of {\it bounded index }\cite{lepson} and in this case we denote $N(f):=N(f,1).$

In 1970 W. J. Pugh and S. M. Shah  \cite{pughshah1970} posed some questions about properties of entire functions of bounded index. One of those questions is following:
\textit{I.  What are the growth properties of functions of bounded index: (c) is it possible to derive the boundedness (or the unboundedness) of the index from the asymptotic properties of the logarithm of the maximum modulus of $f(z),$ i.e., $\ln M(r, f)$?} 

 W. K. Hayman \cite{Hayman} proved that entire function of bounded index has exponential type which is not greater than $N(f)+1.$ 
Later A. D. Kuzyk and M. M. Sheremeta \cite{vidlindex} obtained growth estimate of entire function of bounded $l-$index. M. M. Sheremeta \cite{sher}, T. O. Banakh and V. O. Kushnir \cite{bankush} deduced analogical inequalities for analytic in  a unit disc and in arbitrary complex domain  function of bounded $l$-index, respectively. 

Clearly, the question of Shah and Pugh can be formulated for entire in $\mathbb{C}^n$ function:
\textit{ What are the growth properties of functions of bounded $\mathbf{L}$-index in joint variables? Is it possible to derive the boundedness (or the unboundedness) of the $\mathbf{L}$-index in joint variables from the asymptotic properties of the logarithm of the maximum modulus of $F(z)$ on a skeleton in a polydisc?} 

M. T. Bordulyak and M. M. Sheremeta \cite{bagzmin} gave an answer to the question if $\mathbf{L}(z)=(l_1(|z_1|),\ldots, l_{n}(|z_n|)),$ 
and  for every  $j\in\{1,\ldots,n\}$ the function $l_j: \mathbb{R}_+\to \mathbb{R}_+$ is continuous. 
In this paper we extend their results for $\mathbf{L}(z)=(l_1(z),\ldots, l_{n}(z)),$ where  \ 
$l_j:\mathbb{C}^n\to \mathbb{R}_+$ is a continuous function for every  $j\in\{1,\ldots,n\}$. In  some sense our results are new even in one-dimensional case 
(see below Corollaries \ref{growthcor2joint} and \ref{growthcor3joint}). 

\section{Notations and definitions}
We need some standard notations. Let $\mathbb{R}_+=[0,+\infty)$.
Denote
$$
\mathbf{0}=(0,\ldots,0)\in\mathbb{R}^n_{+},\ \ 
\mathbf{1}=(1,\ldots,1)\in\mathbb{R}^n_{+},\ \ \mathbf{2}=(2,\ldots,2)\in\mathbb{R}^n_{+},$$ 
$$ \mathbf{e}_j=(0,\ldots,0, \underbrace{1}_{j-\mbox{th place}}, 0,\ldots,0)\in\mathbb{R}^n_{+},
\ \ 
[0,2\pi]^n=\underbrace{[0,2\pi]\times \cdots \times[0,2\pi]}_{n-\text{th times}}.
$$
For $R=(r_1,\ldots,r_n)\in\mathbb{R}^n_{+},$ $\Theta=(\theta_1,\ldots,\theta_n)\in[0,2\pi]^n$ and $K=(k_1,\ldots,k_n)\in \mathbb{Z}^n_{+}$ let us to denote
$\displaystyle\|R\|=r_1+\cdots+r_n,$\ $Re^{i\Theta}=(r_1e^{i\theta_1},\ldots,r_ne^{i\theta_n}),$ 
$K!=k_1!\cdot \ldots \cdot k_n!.$
For $A=(a_1,\ldots,a_n)\in\mathbb{C}^n,$ $B=(b_1,\ldots,b_n)\in\mathbb{C}^n,$   we will use formal notations 
without violation of the existence of these expressions 
\begin{gather*} 
|A|=(|a_1|,\ldots,|a_n|), \ 
A\pm B=(a_1\pm b_1,\ldots,a_n\pm b_n), \ 
AB=(a_1b_1,\cdots,a_nb_n),\\
\mathop{arg} A=(\mathop{arg}a_1,\ldots,\mathop{arg}a_n),
A/B=(a_1/b_1,\ldots,a_n/b_n),\ \
A^B=a_1^{b_1}a_2^{b_2}\cdot \ldots a_n^{b_n},\  
\end{gather*}
and a notation $A<B$ means that $a_j<b_j$ for all $j\in\{1,\ldots,n\};$ similarly, the relation $A\leq B$ is defined.

The polydisc $\{z\in\mathbb{C}^n: \ |z_j-z_j^0|<r_j, \ j=1,\ldots, n\}$ is denoted by $D^n(z^0,R),$ its skeleton $\{z\in\mathbb{C}^n: \ |z_j-z_j^0|=r_j, \ j=1,\ldots, n\}$ is denoted by $T^n(z^0,R),$ and the closed polydisc
$\{z\in\mathbb{C}^n: \ |z_j-z_j^0|\leq r_j, \ j=1,\ldots, n\}$ is denoted by $D^n[z^0,R].$
For a partial derivative of entire function $F(z)=F(z_1,\ldots,z_n)$ we will use the notation
$$
F^{(K)}(z)=\frac{\partial^{\|K\|} F}{\partial z^{K}}= \frac{\partial^{k_1+\cdots+k_n}f}{\partial z_1^{k_1}\ldots \partial z_n^{k_n}},  
\text{ where } K=(k_1,\ldots,k_n)\in \mathbb{Z}^n_{+}.$$

Let $\mathbf{L}(z)=(l_1(z),\ldots, l_{n}(z)),$ where $l_j(z)$ are positive continuous functions of variable  $z\in\mathbb{C}^n,$ $j\in\{1,2,\ldots,n\}.$

		An entire function $F(z)$ is called
  {\it a function of bounded $\mathbf{L}$-index in joint variables,} \cite{sufjointdir,monograph} if there exists
  a number $m\in\mathbb{Z}_{+}$ such that for all
  $z\in\mathbb{C}^{n}$ and $J=(j_{1},j_{2},\ldots , j_{n})\in\mathbb{Z}^{n}_{+}$
  \begin{equation} \label{ineqoz2}
  \frac{|F^{(J)}(z)|}{J!\mathbf{L}^{J}(z)}\leq\max
  \left\{\frac{|F^{(K)}(z)|}{K! \mathbf{L}^{K}(z)}:\
  K\in\mathbb{Z}^{n}_{+},\ \|K\|\leq m\right\}.
  \end{equation}
  
  If $l_j=l_j(|z_j|)$ then we obtain a concept of entire functions of bounded $\mathbf{L}$-index in sense of definition in the papers
  \cite{bagzmin,prostir}.
  If $l_{j}(z_{j})\equiv 1,\ 
  j\in\{1,2,\ldots,n\},$ then the entire function is called a {\it function of
   bounded index in joint variables} \cite{krishna,indsalmassi,nuraypattersonmultivalence2015}.
  
  The least integer $m$ for which  inequality holds is called $\mathbf{L}${\it-index in joint variables of the function} $F$ and is denoted by
  $N(F,\mathbf{L}).$

  For $R\in\mathbb{R}^n_{+},$ $j\in\{1,\ldots,n\}$ and $\mathbf{L}(z)=(l_1(z),\ldots,l_n(z))$ we define
  \begin{gather*}
    \lambda_{1,j}(z_0,R)=\inf\left\{\frac{l_j(z)}{l_j(z^0)}\colon z\in D^n\left[z^0,\frac{R}{\mathbf{L}(z^0)}\right] \right\}, \ \ 
  \lambda_{1,j}(R)= \inf_{z^0\in\mathbb{C}^n}   \lambda_{2,j}(z_0,R), \\
  \lambda_{2,j}(z_0,R)=\sup\left\{\frac{l_j(z)}{l_j(z^0)}\colon z\in D^n\left[z^0,\frac{R}{\mathbf{L}(z^0)}\right] \right\}, \ \ 
    \lambda_{2,j}(R)= \sup_{z^0\in\mathbb{C}^n}   \lambda_{2,j}(z_0,R), 
  \\
  \Lambda_1(R)=(\lambda_{1,j}(R),\ldots,\lambda_{1,n}(R)), \ \  \Lambda_2(R)=(\lambda_{2,1}(R),\ldots,\lambda_{2,n}(R)).
  \end{gather*}
  
  By $Q^n$ we denote a class of functions $\mathbf{L}(z)$ which  for
  every $R\in\mathbb{R}^n_{+}$ and $j\in\{1,\ldots,n\}$ satisfy the condition 
\begin{equation} \label{classqn}
  0<\lambda_{1,j}(R)\leq \lambda_{2,j}(R)<+\infty.
  \end{equation}

\section{Auxiliary propositions}
   \begin{proposition} \label{boundlogderjoint} 
      Let $\mathbf{L}(z)=(l_1(z),\ldots, l_{n}(z)),$ $l_j: \mathbb{C}^n \to \mathbb{C}$ and $\frac{\partial l_j}{\partial z_m}$ be continuous functions in $\mathbb{C}^n$ for all $j,$ $m\in\{1,2,\ldots,n\}.$ 
      If there exist  numbers $P>0$ and $c>0$ such that
      for all $z\in\mathbb{C}^n$ and every $ j, m\in\{1,2,\ldots,n\}$
      \begin{equation} \label{condqnb}
      \frac{1}{c+|l_j(z)|}\left|\frac{\partial l_j(z)}{\partial z_m}\right|\leq P 
      \end{equation}
      then $\mathbf{L}^*\in Q^n,$ where $\mathbf{L}^*(z)=(c+|l_1(z)|,\ldots, c+|l_n(z)|).$
   \end{proposition}	
      \begin{proof}
      Clearly, the function $\mathbf{L}^*(z)$ is positive and continuous.
      For given $z\in\mathbb{C}^n,$ $z^0\in\mathbb{C}^n$ we define an analytic curve $\varphi: [0,1]\to \mathbb{C}^n$
      $$\varphi_j(\tau)=z^0_j+\tau (z_j-z_j^0), \ j\in\{1,2,\ldots,n\},$$
      where $\tau\in [0,1].$ 
     It is known that for every continuously differentiable function  $g$ of real variable $\tau$  the inequality  $\frac{d}{dt}|g(\tau)|\leq |g'(\tau)|$  holds except the points  where $g'(\tau)=0.$
      Using restrictions of this lemma, we establish the upper estimate of $\lambda_{2,j}(z_0,R):$
      \begin{gather*}
      \lambda_{2,j}(z_0,R)=\sup \left\{\frac{c+|l_j(z)|}{c+|l_j(z^0)|}\colon z\in D^n\left[z^0,\frac{R}{\mathbf{L}_1(z^0)}\right]   \right\}= \\ 
      = \sup_{z\in D^n\left[z^0,\frac{R}{\mathbf{L}_1(z^0)}\right]  } \left\{ \exp\left\{\ln(c+|l_j(z)|) - \ln(c+|l_j(z^0)|)\right\}
      \right\} = \\
      \!=\! \sup\! \left\{\exp\left\{ \int_0^{1} \frac{d(c+|l_j(\varphi(\tau))|)}{c+|l_j(\varphi(\tau))|} \right \}: z\in D^n\left[z^0,\frac{R}{\mathbf{L}_1(z^0)}\right]  \! \right\} \!\leq\! \\
      \!\leq\! \sup_{ z\in D^n\left[z^0,\frac{R}{\mathbf{L}_1(z^0)}\right] }\!\left\{\! \exp\!\left\{\!\int_0^{1}\! \sum_{m=1}^n \frac{|\varphi_m'(\tau)|}{c+|l_j(\varphi(\tau))|} 
       \left|\frac{\partial l_j(\varphi(\tau))}{\partial z_m}\right| d\tau \right\}  \!\right \} \!\leq\! \\
      \leq \sup_{  z\in D^n\left[z^0,\frac{R}{\mathbf{L}_1(z^0)}\right] } \left\{ \exp\left\{ \int_0^1 \sum_{m=1}^n P|z_m-z_m^0| d\tau \right\}
      \right\} \leq \\ \leq 
      \sup_{z\in D^n\left[z^0,\frac{R}{\mathbf{L}_1(z^0)}\right] } \left\{ \exp\left\{ \sum_{m=1}^n \frac{Pr_j}{c+|l_m(z^0)|} \right\}
      \right\}
       \leq \exp\left(\frac{P}{c}\sum_{m=1}^nr_j \right).
      \end{gather*}
      Hence,  for all $R\geq \mathbf{0}$ \ $\lambda_{2,j}(R)=\sup\limits_{z^0\in\mathbb{C}^n}  \lambda_{2,j}(z^0,\eta) \leq \exp\left(\frac{P}{c}\sum\limits_{m=1}^nr_j \right) <\infty.$
      Using $\frac{d}{dt}|g(t)|\geq -|g'(t)|$ it can be proved that for every $\eta\geq 0$ \  $\lambda_{1,j}(R)\geq \exp\left(-\frac{P}{c}\sum\limits_{m=1}^nr_j \right)>0.$ Therefore, $\mathbf{L}^*\in Q^n.$
   \end{proof}

For estimate of growth of entire functions of bounded $\mathbf{L}$-index in joint variables we will use the following theorem which describes local behaviour of these entire functions.

  \begin{theorem}[\cite{sufjointdir,monograph}] \label{bordte43}\sl
	Let $\mathbf{L}\in Q^n.$ An entire function $F$ is of bounded $\mathbf{L}$-index in joint variables if and only if there exist numbers $R',$ $R'',$ $\mathbf{0}<R'<\mathbf{e}<R'',$  and $p_1=p_1(R',R'')\geq 1$ such that for every $z^0\in\mathbb{C}^n$ 
	\begin{equation}
	\label{bordriv112}
	\max \left\{|F(z)|\colon  z\in T^n\left(z^0,\frac{R''}{\mathbf{L}(z^0)}\right)\right\} \leq p_1  \max \left\{|F(z)|\colon  z\in T^n\left(z^0,\frac{R'}{\mathbf{L}(z^0)}\right)\right\}. 
	\end{equation}
\end{theorem}

At first we prove the following lemma. 
\begin{lemma}  \label{qngrowth}
  If $\mathbf{L}\in Q^n,$ then  for every $ j\in\{1,\ldots,n\}$ and for every fixed $z^*\in\mathbb{C}^n$ $|z_j|l_j(z^*+z_j\mathbf{e}_j)\to \infty$ as $|z_j|\to \infty.$ 
\end{lemma}
\begin{proof}
   On the contrary, if there exist a number $C>0$ and a sequence $z_j^{(m)}\to\infty$ such that $|z_j^{(m)}|l_j(z^*+z_j^{(m)}\mathbf{e}_j)=k_m\leq C,$ i. e. 
   $|z_j^{(m)}| =\frac{k_m}{l_j(z^*+z_j^{(m)}\mathbf{e}_j)}.$ Then 
   \begin{gather*}
   \frac{1}{l_j(z^*+z_j^{(m)}\mathbf{e}_j)}l_j\left(z^*+z_j^{(m)}\mathbf{e}_j-\frac{k_me^{i\mathop{arg} z_j^{(m)}}\mathbf{e}_j}{l_j(z^*+z_j^{(m)}\mathbf{e}_j)}\right)=
   \frac{|z_j^{(m)}|}{k_m}l_j(z^*)\to +\infty, \ j\to +\infty,
   \end{gather*}
that is $\lambda_{2,j}(C\mathbf{e}_j)=+\infty$ and $\mathbf{L}\notin Q^n.$
\end{proof}

\section{Estimates of growth of entire functions}
  By $K^n$ we denote a class of positive continuous functions $\mathbf{L}(z)$ for which  
  there exists $c>0$  such that for
every $R\in\mathbb{R}^n_{+}$ and $j\in\{1,\ldots,n\}$  
$$
\max_{\Theta_1,\Theta_2\in[0,2\pi]^n} \frac{l_j(Re^{i\Theta_2})}{l_j(Re^{i\Theta_1})}\leq c.
$$
If $\mathbf{L}(z)=(l_1(|z_1|,\ldots,|z_n|),\ldots,l_n(|z_1|,\ldots,|z_n|))$ then $\mathbf{L}\in K^n.$
It is easy to prove that $|e^z| +1\in Q^1\setminus K^1,$ but $e^{e^{|z|}}\in K^1\setminus Q^1,$ $z\in\mathbb{C}.$ 
Besides, if $\mathbf{L}_1, \mathbf{L}_2\in K^n$ then $\mathbf{L}_1+\mathbf{L}_2\in K^n$ and $\mathbf{L}_1\mathbf{L}_2\in K^n.$
For simplicity, let us to write $K\equiv K^1$  and 
$M(F,R)= \max \{|F(z)|\colon z\in T^n(\mathbf{0},R)\}.$
\begin{theorem} \label{thgrowthobig}
   Let $\mathbf{L}\in Q^n\cap K^n.$ If an entire function $F$ has bounded $\mathbf{L}$-index in joint variables, then 
\begin{equation}
\label{growthobig}
      \ln M(F,R) = O\left(\min_{\sigma_n\in \mathcal{S}_n} \min_{\Theta\in[0,2\pi]^n} \sum_{j=1}^n \int_{0}^{r_j} l_j(R(j,\sigma_n,t)e^{i\Theta}) dt\right)
 \text{ as } \|R\|\to \infty,
\end{equation}
where $\sigma_n$ is a permutation of $\{1,\ldots,n\},$ $R(j,\sigma_n,t)=(r'_1,\ldots,r'_n),$  $r'_k=\begin{cases}
r^0_k, \text{ if } \sigma_n(k)<j,\\
t, \text{ if } k=j,\\
r_k, \text{ if } \sigma_n(k)>j,
\end{cases}$
$k\in \{1,\ldots,n\},$ 
$R^0=(r_1^0,\ldots,r_n^0)$ is sufficiently large radius, 
$\mathcal{S}_n$ is a set of all permutations of  $\{1,\ldots,n\}.$
\end{theorem}

\begin{proof}
   Let $R>0,$ $\Theta\in[0,2\pi]^n$ and a point $z^*\in T^n(\mathbf{0},R+\frac{\mathbf{2}}{\mathbf{L}(Re^{i\Theta})})$ be a such that 
   $$|F(z^*)|=\max\left\{|F(z)|: z\in T^n\left(\mathbf{0},R+\frac{\mathbf{2}}{\mathbf{L}(Re^{i\Theta})}\right)\right\}.$$
    Denote $z^0=\frac{z^*R}{R+\mathbf{2}/\mathbf{L}(Re^{i\Theta})}.$ Then 
    \begin{gather*}
    |z^0-z^*|=
    \left|\frac{z^*R}{R+\mathbf{2}/\mathbf{L}(Re^{i\Theta})} -z^*\right| = 
    \left|\frac{z^*\mathbf{2}/\mathbf{L}(Re^{i\Theta})}{R+\mathbf{2}/\mathbf{L}(Re^{i\Theta})} \right| = 
    \frac{\mathbf{2}}{\mathbf{L}(Re^{i\Theta})} \text{ and } \\ 
    \mathbf{L}(z^0)=\mathbf{L} \left(\frac{z^*R}{R+\mathbf{2}/\mathbf{L}(Re^{i\Theta})}\right) = \mathbf{L} 
    \left(\frac{(R+\mathbf{2}/\mathbf{L}(Re^{i\Theta}))e^{i\mathop{arg}z^*}R}{R+\mathbf{2}/\mathbf{L}(Re^{i\Theta})}\right) = 
    \mathbf{L} (Re^{i\mathop{arg} z^*}).
    \end{gather*}
    Since $\mathbf{L}\in K^n$ we have that $c\mathbf{L}(z^0)= c\mathbf{L}(Re^{i\mathop{arg}z^*})\geq \mathbf{L}(Re^{i\Theta}) \geq \frac{1}{c}\mathbf{L}(z^0).$
     We consider two 
   skeletons $T^n(z^0,\frac{\mathbf{1}}{\mathbf{L}(z^0)})$ and $T^n(z^0,\frac{\mathbf{2}}{\mathbf{L}(z^0)}).$ 
   By Theorem \ref{bordte43} there exists $p_1=p_1(\frac{\mathbf{1}}{c},c\mathbf{2})\geq 1$ such that 
   \eqref{bordriv112} holds with $R'=\frac{\mathbf{1}}{c},$ $R''=c\mathbf{2},$ 
   i.e. 
   \begin{gather}
  \! \max\!\left\{|F(z)| \colon z\!\in\! T^n\left(\mathbf{0},R+\frac{\mathbf{2}}{\mathbf{L}(Re^{i\Theta})}\!\right)\!\right\}\!=\!|F(z^*)| \!\leq\! 
   \max\left\{|F(z)| \colon z\!\in\! T^n\left(z^0,\frac{\mathbf{2}}{\mathbf{L}(Re^{i\Theta})}\right)\right\} \! \leq\! \nonumber \\ 
\leq       \max\left\{|F(z)| \colon z\in T^n\left(z^0,\frac{c\mathbf{2}}{\mathbf{L}(z^0)}\right)\right\} \leq 
p_1   \max\left\{|F(z)| \colon z\in T^n\left(z^0,\frac{\mathbf{1}}{c\mathbf{L}(z^0)}\right)\right\} \leq \nonumber \\ 
\leq p_1\max\left\{|F(z)|: z\in T^n\left(\mathbf{0},R+\frac{\mathbf{e}}{\mathbf{L}(Re^{i\Theta})}\right)\right\}
    \label{bordriv22}
   \end{gather}
   A function $\ln^+ \max \{|F(z)|\colon z\in T^n(\mathbf{0},R)\}$ is a convex function of the variables  $\ln{r_1},$ $\ldots,$ $\ln{r_n}$ (see \cite{ronkin}, p.~138 in Russian edition or p.~84 in English translation). Hence, the function  admits a representation  
   \begin{gather}
   \ln^+ \max \{|F(z)|\colon z\in T^n(\mathbf{0},R)\} -
   \ln^+ \max \{|F(z)|\colon z\in T^n(\mathbf{0},R+(r_j^0-r_j)\mathbf{e}_j)\}  = \nonumber\\ =
   \int_{r_j^0}^{r_j} \frac{A_j(r_1,\ldots,r_{j-1},t,r_{j+1},\ldots,r_n)}{t}dt \label{presentconvex}
   \end{gather}
   for arbitrary $0<r_j^0\leq r_j<+\infty,$ where the function $A_j(r_1,\ldots,r_{j-1},t,r_{j+1},\ldots,r_n)$ is a positive non-decreasing in variable $t\in(0;+\infty),$ 
    $j\in\{1, \ldots,n\}.$
   
   Using \eqref{bordriv22} we deduce 
   \begin{gather}
   \ln p_1 \ge \ln\max\left\{|F(z)| \colon z\in T^n\left(\mathbf{0},R+\frac{\mathbf{2}}{\mathbf{L}(Re^{i\Theta})}\right)\right\} -\nonumber 
   \\ -\ln \max\left\{|F(z)|: z\in T^n\left(\mathbf{0},R+\frac{\mathbf{e}}{\mathbf{L}(Re^{i\Theta})}\right)\right\}=\nonumber\\
   =\sum_{j=1}^n \ln\max\left\{|F(z)| \colon z\in T^n\left(\mathbf{0},R+\frac{\mathbf{1}+\sum_{k=j}^n \mathbf{e}_k}{\mathbf{L}(Re^{i\Theta})}\right)\right\} -\nonumber\\ - 
   \ln\max\left\{|F(z)| \colon z\in T^n\left(\mathbf{0},R+\frac{\mathbf{1}+\sum_{k=j+1}^n \mathbf{e}_k}{\mathbf{L}(Re^{i\Theta})}\right)\right\}=\nonumber 
   \\ 
   = \sum_{j=1}^n \int_{r_j+1/l_j(Re^{i\Theta})}^{r_j+2/l_j(Re^{i\Theta})} \frac{1}{t} 
   A_j\left(r_1+\frac{1}{l_1(Re^{i\Theta})},\ldots, r_{j-1}+\frac{1}{l_{j-1}(Re^{i\Theta})},t, 
   r_{j+1}+\frac{2}{l_1(Re^{i\Theta})},\ldots, \right.\nonumber\\  \left.
   r_n+\frac{2}{l_n(Re^{i\Theta})}\right)dt  
   \!\geq\! 
   \sum_{j=1}^n \ln\left(1+\frac{1}{r_jl_j(Re^{i\Theta})\!+\!1}\right)  
      A_j\!\left(r_1\!+\!\frac{1}{l_1(Re^{i\Theta})},\ldots, r_{j-1}\!+\!\frac{1}{l_{j-1}(Re^{i\Theta})},r_j, \nonumber\right. \\ \left.
   r_{j+1}+\frac{2}{l_1(Re^{i\Theta})},\ldots, r_n+\frac{2}{l_n(Re^{i\Theta})}\right) \label{estimaconvex}
      \end{gather}
   By Lemma \ref{qngrowth} the function  $r_jl_j(Re^{i\Theta})\to +\infty$ $(r_j\to +\infty).$ Hence, for $j\in\{1,\ldots, n\}$ and $r_i\geq r_i^0$  
    $$\ln\left(1+\frac{1}{r_jl_i(Re^{i\Theta})+1}\right) \sim \frac{1}{r_jl_j(Re^{i\Theta})+1} \geq \frac{1}{2r_jl_j(Re^{i\Theta})}.$$ 
    Thus, for every $j\in\{1,\ldots, n\}$ inequality \eqref{estimaconvex}  implies that 
   \begin{gather*}
   A_j\left(r_1+\frac{1}{l_1(Re^{i\Theta})},\ldots,r_{i-1}+\frac{1}{l_{j-1}(Re^{i\Theta})},r_j,r_{j+1}+\frac{2}{l_{i+1}(Re^{i\Theta})},
   \ldots, r_n+\frac{2}{l_n(Re^{i\Theta})} \right) \leq \\ \leq
   2\ln p_1\ r_j l_j(Re^{i\Theta}).
   \end{gather*}
Let $R^0=(r_1^0,\ldots,r_n^0),$ where every $r_j^0$ is above chosen.
  Applying \eqref{presentconvex} $n$-th times consequently  we obtain 
   \begin{gather*}
      \ln \max \{|F(z)|\colon z\in T^n(\mathbf{0},R)\} =   \ln \max \{|F(z)|\colon z\in T^n(\mathbf{0},R+(r_1^0-r_1)\mathbf{e}_1)\}  + \\ +
      \int_{r_1^0}^{r_1} \frac{A_1(t,r_2,\ldots,r_n)}{t}dt= 
       \ln \max \{|F(z)|\colon z\in T^n(\mathbf{0},R+(r_1^0-r_1)\mathbf{e}_1+(r_2^0-r_2)\mathbf{e}_2)\}  + \\ +
  \int_{r_1^0}^{r_1} \frac{A_1(t,r_2,\ldots,r_n)}{t}dt+  \int_{r_2^0}^{r_2} \frac{A_2(r_1^0,t,r_3\ldots,r_n)}{t}dt 
              =  \ln \max \{|F(z)|\colon z\in T^n(\mathbf{0},R^0)\} + \\ 
       +\sum_{j=1}^n \int_{r_j^0}^{r_j} \frac{A_j(r_1^0,\ldots, r_{j-1}^0,t,r_{j+1},\ldots,r_n)}{t} dt       
             \leq \ln \max \{|F(z)|\colon z\in T^n(\mathbf{0},R^0)\} + \\        
+   2\ln p_1\sum_{j=1}^n \int_{r_j^0}^{r_j} l_j(r_1^0e^{i\theta_1},\ldots, r_{j-1}^0e^{i\theta_{j-1}},te^{i\theta_j},r_{j+1}e^{i\theta_{j+1}},\ldots,r_n  e^{i\theta_n}) dt \leq 
       \\
           \leq \ln \max \{|F(z)|\colon z\in T^n(\mathbf{0},R^0)\} + \\        
       +   2\ln p_1\sum_{j=1}^n \int_{0}^{r_j} l_j(r_1^0e^{i\theta_1},\ldots, r_{j-1}^0e^{i\theta_{j-1}},te^{i\theta_j},r_{j+1}e^{i\theta_{j+1}},\ldots,r_n  e^{i\theta_n}) dt \leq 
       \\ \leq 
   (1+o(1)) 2\ln p_1 \sum_{j=1}^n \int_{0}^{r_j} l_j(r_1^0e^{i\theta_1},\ldots, r_{j-1}^0e^{i\theta_{j-1}},te^{i\theta_j},r_{j+1}e^{i\theta_{j+1}},\ldots,r_n  e^{i\theta_n}) dt.
   \end{gather*}
   The function $\ln \max \{|F(z)|\colon z\in T^n(\mathbf{0},R)\}$ is independent of $\Theta.$  Thus,  the following estimate holds
   \begin{gather*}  
         \ln \max \{|F(z)|\colon z\in T^n(\mathbf{0},R)\} = \\ \!=\!O\left(\!\min_{\Theta\in[0,2\pi]^n}\sum_{j=1}^n \int_{0}^{r_j} l_j(r_1^0e^{i\theta_1},\ldots, r_{j-1}^0e^{i\theta_{j-1}},te^{i\theta_j},r_{j+1}e^{i\theta_{j+1}},\ldots,r_n  e^{i\theta_n}) dt\right),
         \text{ as } \|R\| \to +\infty.
  \end{gather*}
   It is obviously that similar equality can be proved for arbitrary permutation $\sigma_n$ of the set $\{1,2,\ldots,n\}.$ 
    Thus, estimate \eqref{growthobig} holds.
   Theorem \ref{thgrowthobig} is proved.
\end{proof}

\begin{corollary}
	If 	$\mathbf{L}\in Q^n\cap K^n,$ $\min\limits_{\Theta\in[0,2\pi]^n}l_j(Re^{i\Theta})$ is non-decreasing in each variable $r_k,$ $k\in\{1,\ldots,n\},$ 
  entire function $F$ has bounded $\mathbf{L}$-index in joint variables then 
  \begin{equation*}
  \ln \max \{|F(z)|\colon z\in T^n(\mathbf{0},R)\} = O\left(\min_{\Theta\in[0,2\pi]^n} \sum_{j=1}^n \int_{0}^{r_j} l_j(R^{(j)}e^{i\Theta}) dt\right)
  \ \text{ as } \|R\|\to \infty,
  \end{equation*}
  where $R^{(j)}=(r_1,\ldots,r_{j-1},t,r_{j+1},\ldots,r_n).$
\end{corollary}

Note that Theorem \ref{thgrowthobig} is new too for $n=1$ because we replace the condition $l=l(|z|)$ by the condition $l\in K,$ i.e. 
 there exists $c>0$  such that for
every $r>0$ 
$\max\limits_{\theta_1,\theta_2\in[0,2\pi]} \frac{l(re^{i\theta_2})}{l(re^{i\theta_1})}\leq c.$ 
Particularly, the following proposition is valid. 
\begin{corollary} \label{growthcor2joint}
	If 	$l\in Q\cap K$ and an entire in $\mathbb{C}$ function $f$ has bounded $l$-index then 
\begin{equation*}
\ln \max \{|f(z)|\colon |z|=r\} = O\left(\min_{\theta\in[0,2\pi]}  \int_{0}^{r} l(te^{i\Theta}) dt\right)
\ \text{ as } r\to \infty.
\end{equation*}
\end{corollary}

W. K. Hayman, A. D. Kuzyk, M M. Sheremeta, V. O.  Kushnir and T. O. Banakh \cite{Hayman,vidlindex,bankush} improved an estimate \eqref{growthobig} 
 by other conditions on the function $l$ for a case $n=1.$ M. T. Bordulyak and M. M. Sheremeta \cite{bagzmin} deduced similar results for entire functions of bounded $\mathbf{L}$-index in joint variables, if $l_j=l_j(|z_j|),$ $j\in\{1,\ldots,n\}.$
Using their method   we will generalise the estimate for $l_j:\mathbb{C}^n\to \mathbb{R}_+.$   

Let us to denote $a^+=\max\{a,0\},$  $u_j(t)=u_j(t,R,\Theta)=l_j(\frac{tR}{r_m}e^{i\Theta}),$   where $a\in\mathbb{R},$ $t\in\mathbb{R}_+,$ $j\in\{1,\ldots,n\},$ 
$r_m\neq 0.$
\begin{theorem} 
   \label{bordte22}
   Let $\mathbf{L}(Re^{i\Theta})$  be a positive continuously differentiable function in each variable $r_k\in[0,+\infty),$ $k\in\{1,\ldots,n\},$  $\Theta\in[0,2\pi]^n.$  If an entire function $F$ has bounded $\mathbf{L}$-index $N=N(F,\mathbf{L})$ in joint variables then for every 
   $\Theta\in[0,2\pi]^n$ and for every $R\in\mathbb{R}^n_+$   ($r_m\neq 0$)  and $S\in\mathbb{Z}^n_+$
\begin{gather}
\ln\max\left\{\frac{|F^{(K)}(Re^{i\Theta})|}{K!\mathbf{L}^K(Re^{i\Theta})}:\ \|K\|\leq N \right\}\leq 
\ln\max\left\{\frac{|F^{(K)}(0)|}{K!\mathbf{L}^K(0)}:\ \|K\|\leq N \right\}+ \nonumber \\ +
\int_0^{r_m} \left(\max_{\|K\|\leq N}\left\{
\sum_{j=1}^n \frac{r_j}{r_m}(k_j+1)l_j\left(\frac{\tau}{r_m} Re^{i\Theta}\right) 
\right\} + \max_{\|K\|\leq N}\left\{ \sum_{j=1}^n \frac{k_j(-u_j'(\tau))^+}{l_j\left(\frac{\tau}{r_m} Re^{i\Theta}\right)}  \right\}\right)d\tau,
\label{conclusiongrowth}
\end{gather}   

If, in addition,  there exists $C>0$ such that the function $\mathbf{L}$ satisfies inequalities
\begin{gather} \label{suplinf}
\sup\limits_{R\in\mathbb{R}^n_+} \max\limits_{t\in[0,r_m]}\max\limits_{\Theta\in[0,2\pi]^n} \max\limits_{1\leq j\leq n}\frac{(-(u_j(t,R,\Theta))'_t)^+}{\frac{r_j}{r_m} l_j^2(\frac{t}{r_m}Re^{i\Theta})} \leq C,
\\ \label{intliminf}
\max\limits_{\Theta\in[0,2\pi]^n} \int_0^{r_m}  \sum_{j=1}^n \frac{r_j}{r_m}l_j\left(\frac{\tau}{r_m}R e^{i\Theta}\right) d\tau\to +\infty 
\text{ as } \|R\|\to +\infty
\end{gather}
 then 
\begin{equation}
\label{nonzerogrowth}
\varlimsup_{\|R\|\to +\infty}   \frac{\ln \max\{|F(z)\colon \ z\in T^n(\mathbf{0},R)\}}{\max\limits_{\Theta\in[0,2\pi]^n} \int_0^{r_m}  \sum_{j=1}^n \frac{r_j}{r_m}l_j\left(\frac{\tau}{r_m}R e^{i\Theta}\right) d\tau}\leq (C+1) N+1.
\end{equation}

And if $r_m(-(u_j(t,R,\Theta))'_t)^+/(r_j l_j^2(\frac{t}{r_m}Re^{i\Theta}))\to 0$ and \eqref{intliminf} holds as $\|R\|\to +\infty$ 
uniformly for all $\Theta\in[0,2\pi]^n,$ $j\in\{1,\ldots,n\},$ $t\in[0,r_m]$  then 
\begin{equation}
\label{zerogrowth}
\varlimsup_{\|R\|\to +\infty}   \frac{\ln \max\{|F(z)\colon \ z\in T^n(\mathbf{0},R)\}}{\max\limits_{\Theta\in[0,2\pi]^n} \int_0^{r_m}  \sum_{j=1}^n \frac{r_j}{r_m}l_j\left(\frac{\tau}{r_m}R e^{i\Theta}\right) d\tau}\leq N+1.
\end{equation}
\end{theorem}
\begin{proof}
   Let $R\in \mathbb{R}\setminus \{\square\},$ $\Theta\in[0,2\pi]^n.$ Then there exists at least one $r_m\neq 0.$ Denote $\alpha_j=\frac{r_j}{r_m},$ $j\in\{1,\ldots,n\}$ and $A=(\alpha_1,\ldots,\alpha_n).$  
    We consider a function 
   \begin{equation} \label{defingt}
   g(t)=\max\left\{\frac{|F^{(K)}(Ate^{i\Theta})|}{K!\mathbf{L}^K(Ate^{i\Theta})}:\ \|K\|\leq N \right\},
   \end{equation}
   where $At=(\alpha_1t,\ldots,\alpha_nt),$ $Ate^{i\Theta}= (\alpha_1te^{i\theta_1},\ldots,\alpha_nte^{i\theta_n}).$  
   
   Since the function $\frac{|F^{(K)}(Ate^{i\Theta})|}{K!\mathbf{L}^{K}(Ate^{i\Theta})}$ is  continuously differentiable by real $t\in[0,+\infty),$
  outside the zero set of function $|F^{(K)}(Ate^{i\Theta})|,$
  the function $g(t)$ is a continuously differentiable function on $[0,+\infty),$ except, perhaps, for a countable set of points.
   
 Therefore, using    the inequality
$\frac{d}{dr} |g(r)|\leq |g'(r)|$  which holds except for the points $r=t$ such that $g(t)=0,$  we deduce 
\begin{gather}
\frac{d}{dt} \left(\frac{|F^{(K)}(Ate^{i\Theta})|}{K!\mathbf{L}^K(Ate^{i\Theta})}\right)= \frac{1}{K!\mathbf{L}^K(Ate^{i\Theta})} 
\frac{d}{dt}|F^{(K)}(Ate^{i\Theta})|+|F^{(K)}(Ate^{i\Theta})| \frac{d}{dt}\frac{1}{K!\mathbf{L}^K(Ate^{i\Theta})} \leq \nonumber \\
\leq  \frac{1}{K!\mathbf{L}^K(Ate^{i\Theta})} \left|\sum_{j=1}^n F^{(K+\mathbf{e}_j)}(Ate^{i\Theta}) \alpha_j e^{i\theta_j}\right|-
\frac{|F^{(K)}(Ate^{i\Theta})|}{K!\mathbf{L}^K(Ate^{i\Theta})} 
\sum_{j=1}^n  \frac{k_j u'_j(t)}{l_j(Ate^{i\Theta})} \leq \nonumber\\
\leq \sum_{j=1}^n  \frac{|F^{(K+\mathbf{e}_j)}(Ate^{i\Theta})|}{(K+\mathbf{e}_{j})!\mathbf{L}^{K+\mathbf{e}_j}(Ate^{i\Theta})} 
\alpha_j(k_j+1)l_j(Ate^{i\Theta})+ 
\frac{|F^{(K)}(Ate^{i\Theta})|}{K!\mathbf{L}^K(Ate^{i\Theta})} \sum_{j=1}^n  \frac{k_j (-u'_j(t))^+}{l_j(Ate^{i\Theta})} \label{derivindex}
\end{gather}
 For absolutely continuous functions
$h_1,$ $h_2, $ $\ldots,$ $h_k$ and $h(x):=\max\{h_j(z): 1\leq j\leq k\},$
\  $h'(x)\leq \max\{h'_j(x): 1\leq j\leq k \},$ $x\in[a,b]$  (see \cite[Lemma~4.1, p.~81]{sher}). The function $g$ is absolutely continuous, therefore, from \eqref{derivindex} it follows that
   \begin{gather*}
   g'(t) \leq \max \left\{\frac{d}{dt} \left(\frac{|F^{(K)}(Ate^{i\Theta})|}{K!\mathbf{L}^K(Ate^{i\Theta})}\right)\colon 
   \|K\|\leq N \right\} \leq \\ \leq 
      \max_{ \|K\|\leq N} \left\{ 
\sum_{j=1}^n  \frac{\alpha_j(k_j+1)l_j(Ate^{i\Theta})|F^{(K+\mathbf{e}_j)}(Ate^{i\Theta})|}{(K+\mathbf{e}_{j})!\mathbf{L}^{K+\mathbf{e}_j}(Ate^{i\Theta})} 
+ 
\frac{|F^{(K)}(Ate^{i\Theta})|}{K!\mathbf{L}^K(Ate^{i\Theta})} \sum_{j=1}^n  \frac{k_j (-u'_j(t))^+}{l_j(Ate^{i\Theta})}    \right\}\leq \\
   \leq g(t) \left( \max_{ \|K\|\leq N}\left\{ 
   \sum_{j=1}^n \alpha_j(k_j+1)l_j(Ate^{i\Theta})\right\}+ 
   \max_{ \|K\|\leq N}\left\{ \sum_{j=1}^n \frac{k_j(-u_j'(t))^+}{l_j(Ate^{i\Theta})}\right\}\right)=\\ =
   g(t)(\beta(t)+\gamma(t)),
   \end{gather*}
   where 
   $   \beta(t)=\max_{\|K\|\leq N}\left\{
   \sum_{j=1}^n \alpha_j(k_j+1)l_j(Ate^{i\Theta}) \right\},$
   $\gamma(t)=\max_{\|K\|\leq N}\left\{ \sum_{j=1}^n \frac{k_j(-u_j'(t))^+}{l_j(Ate^{i\Theta})} \right\}.
   $
   Thus, 
   $\frac{d}{dt} \ln g(t)\leq \beta(t)+\gamma(t)$
   and 
   \begin{equation} \label{deqr1}
  g(t)\leq g(0)\exp \int_0^t (\beta(\tau)+\gamma(\tau))d\tau, 
  \end{equation} because $g(0)\neq 0.$
 But $r_mA=R.$ Substituting $t=r_m$ in \eqref{deqr1} and taking  into account \eqref{defingt},
  we deduce 
\begin{gather*}
\ln\max\left\{\frac{|F^{(K)}(Re^{i\Theta})|}{K!\mathbf{L}^K(Re^{i\Theta})}:\ \|K\|\leq N \right\}\leq 
\ln\max\left\{\frac{|F^{(K)}(0)|}{K!\mathbf{L}^K(0)}:\ \|K\|\leq N \right\}+ \\ +
\int_0^{r_m} \left(\max_{\|K\|\leq N}\left\{
\sum_{j=1}^n \alpha_j(k_j+1)l_j(A\tau e^{i\Theta}) 
\right\} + \max_{\|K\|\leq N}\left\{ \sum_{j=1}^n \frac{k_j(-u_j'(\tau))^+}{l_j(A\tau e^{i\Theta})} \right\}\right)d\tau,
\end{gather*}   
   i.e.  \eqref{conclusiongrowth} is proved. 
Denote $\widetilde{\beta}(t)= \sum_{j=1}^n \alpha_jl_j(Ate^{i\Theta}).$   
   If, in addition, \eqref{suplinf}-\eqref{intliminf} hold  
   then  for some $K^*,$ $\|K^*\|\leq N$ and $\widetilde{K},$ $\|\widetilde{K}\|\leq N,$ 
   \begin{gather*}
\frac{\gamma(t)}{\widetilde{\beta}(t)}= \frac{\sum_{j=1}^n \frac{k^*_j(-u_j'(t))^+}{l_j(Ate^{i\Theta})}}{\sum_{j=1}^n \alpha_jl_j(Ate^{i\Theta})}\leq 
\sum_{j=1}^n k^*_j\frac{(-u_j'(t))^+}{\alpha_jl^2_j(Ate^{i\Theta})}\leq 
\sum_{j=1}^n  k^*_j \cdot C \leq NC \text{ and } \\
\frac{\beta(t)}{\widetilde{\beta}(t)}=\frac{ \sum_{j=1}^n \alpha_j(\tilde{k}_j+1)l_j(Ate^{i\Theta})}{\sum_{j=1}^n \alpha_jl_j(Ate^{i\Theta})} =
1+\frac{ \sum_{j=1}^n \alpha_j\tilde{k}_jl_j(Ate^{i\Theta})}{\sum_{j=1}^n \alpha_jl_j(Ate^{i\Theta})}
 \leq 1+
\sum_{j=1}^n \tilde{k}_j\leq 1+ N.
   \end{gather*}
   
   But 
   $|F(Ate^{i\Theta})|\leq g(t)\leq g(0) \exp \int_0^t (\beta(\tau)+\gamma(\tau))d\tau$ 
    and $r_mA=R.$ Then we put $t=r_m$ 
   and obtain 
   \begin{gather*}
   \ln \max\{|F(z)\colon \ z\in T^n(\mathbf{0},R)\}=\ln \max_{\Theta\in[0,2\pi]^n}|F(Re^{i\Theta})| \leq \ln \max_{\Theta\in[0,2\pi]^n} g(r_m) \leq 
   \\ \leq 
   \ln g(0)+
  \max_{\Theta\in[0,2\pi]^n} \int_0^{r_m} (\beta(\tau)+\gamma(\tau))d\tau   \leq\\ 
   \leq \ln g(0)+(NC+N+1)    \max_{\Theta\in[0,2\pi]^n} \int_0^{r_m} \widetilde{\beta}(\tau) d\tau = \\ 
  = \ln g(0)+(NC+N+1) \max_{\Theta\in[0,2\pi]^n}  \int_0^{r_m}  \sum_{j=1}^n \alpha_jl_j(A\tau e^{i\Theta}) d\tau =\\ 
  =\ln g(0)+(NC+N+1) \max_{\Theta\in[0,2\pi]^n} \int_0^{r_m}  \sum_{j=1}^n \frac{r_j}{r_m}l_j\left(\frac{\tau}{r_m}R e^{i\Theta}\right) d\tau.
   \end{gather*}
   Thus, we conclude that \eqref{nonzerogrowth} holds.
   Estimate \eqref{zerogrowth} can be deduced by analogy. 
   Theorem \ref{bordte22} is proved.
\end{proof}
We will write $u(r,\theta)=l(re^{i\theta}).$  Theorem \ref{bordte22} implies the following proposition for $n=1.$
\begin{corollary} \label{growthcor3joint}
	   Let $l(re^{i\theta})$  be a positive continuously differentiable function in variable $r\in[0,+\infty)$ for every $\theta\in[0,2\pi].$  If an entire function $f$ has bounded $l$-index $N=N(f,l)$ 
		and  there exists $C>0$ such that 
	$\varlimsup\limits_{r\to +\infty}\max\limits_{\theta\in[0,2\pi]} \frac{(-u'_r(r,\theta))^+}{l^2(re^{i\Theta})} = C$   then 
	\begin{equation*}
	\varlimsup_{r\to +\infty}   \frac{\ln \max\{|f(z)\colon  |z|=r\}}{\max\limits_{\theta\in[0,2\pi]} \int_0^{r}   
 l\left(\tau e^{i\theta}\right) d\tau}\leq (C+1)N+1.
	\end{equation*}
\end{corollary}

\begin{remark}
	Our result is sharper than known result of Sheremeta which is obtained in a case $n=1,$ $C\neq 0$ 
	and $l=l(|z|).$
		Indeed, corresponding theorem \cite[p.~83]{sher} claims that 
		\begin{equation*}
	\varlimsup_{r\to +\infty}   \frac{\ln \max\{|f(z)\colon  |z|=r\}}{\int_0^{r}   
		l(\tau) d\tau}\leq (C+1)(N+1).
	\end{equation*}	
	Obviously, that $NC+N+1 < (C+1)(N+1)$ for $C\neq 0$ and $N\neq 0.$
\end{remark}

Estimate \eqref{zerogrowth} is sharp. It is easy to check for function $F(z_1,z_2)=\exp(z_1z_2),$  $l_1(z_1,z_2)=|z_2|+1,$  $l_2(z_1,z_2)=|z_1|+1.$ 
Then $N(F,\mathbf{L})=0$ and $\ln \max\{|F(z)|\colon z\in T^2(\mathbf{0},R)\}=r_1r_2.$


Department of Advanced Mathematics

Ivano-Frankivs'k National Technical University of Oil and Gas

andriykopanytsia@gmail.com

Department of Function Theory and Theory of Probability

Ivan Franko National University of Lviv

olskask@gmail.com


\end{document}